\documentclass{birkjour}
\usepackage{amsmath, amsthm, amscd, amsfonts,amssymb, graphicx}
\usepackage{booktabs}
 \newtheorem{theorem}{Theorem}[section]
 \newtheorem{corollary}[theorem]{Corollary}
 \newtheorem{lemma}[theorem]{Lemma}
 \newtheorem{example}[theorem]{Example}
 \newtheorem{proposition}[theorem]{Proposition}
 \theoremstyle{definition}
 \newtheorem{definition}[theorem]{Definition}
 
 \theoremstyle{remark}
 \newtheorem{remark}[theorem]{Remark}
 \numberwithin{equation}{section}
\usepackage{color}
\usepackage{color}
\usepackage{graphicx}
\usepackage{caption}
\usepackage{subcaption}

\usepackage{float}
\usepackage[linesnumbered,lined,boxed,commentsnumbered]{algorithm2e}
\usepackage{matlab-prettifier}
\SetKwInput{KwData}{Input}
\SetKwInput{KwResult}{Output}

\begin{document}

%-------------------------------------------------------------------------
% editorial commands: to be inserted by the editorial office
%
%\firstpage{1} \volume{228} \Copyrightyear{2004} \DOI{003-0001}
%\seriesextraline{This is the Concrete Title of this Book\br H.E. R and S.T.C. W, Eds.}
%\firstpage{1}

\title[Continuous Approach]
 {Continuous Approach to Phase (Norm) Retrieval Frames}

%----------Author 1
\author[R. Farshchian]{Ramin Farshchian}
\address{Department of Mathematical sciences,   Ferdowsi University of Mashhad,Mashhad, Iran.}
\email{ramin.farshchian@mail.um.ac.ir}

\author[R. A.  Kamyabi-Gol]{Rajab Ali  Kamyabi-Gol}
\address{Department of Mathematical sciences, Faculty of Math, Ferdowsi University of Mashhad and Center of Excellence in Analysis on Algebraic Structures (CEAAS), Mashhad, Iran.}
\email{kamyabi@um.ac.ir}

%----------Author 3
\author[F. Arabyani]{Fahimeh Arabyani-Neyshaburi}
\address{Department of Mathematical sciences,   Ferdowsi University of Mashhad,Mashhad, Iran.}
\email{fahimeh.arabyani@gmail.com}

%----------Author 4
\author[F. Esmaeelzadeh]{Fatemeh Esmaeelzadeh}
\address{Department of Mathematics, Bojnourd Branch, Islamic Azad University, Bojnourd, Iran}
\email{esmaeelzadeh@bojnourdiau.ac.ir}

\vspace{-2.5cm}

%----------------------------------------------
\vspace{1.7cm}

\begin{abstract}

This paper investigates the properties of continuous frames, with a particular focus on phase retrieval and norm retrieval in the context of Hilbert spaces. We introduce the concept of continuous near-Riesz bases and prove their invariance under invertible operators. Some equivalent conditions for  phase and norm retrieval property of continuous frames are presented. We study the stability of phase retrieval under perturbations. Furthermore, tensor product frames for separable Hilbert spaces are studied, and we establish the equivalence of phase retrieval and norm retrieval properties between components and their tensor products. 

\end{abstract}

%%% ----------------------------------------------------------------------
\maketitle

%%% ----------------------------------------------------------------------
%\tableofcontents
%\textbf{Key words: Reconstruction error, signal analysis, Gram matrix, $K$-dual frame, maximal robust, minimal redundancy condition, uniform excess}
%%% ----------------------------------------------------------------------
\maketitle
%%% ----------------------------------------------------------------------
%\tableofcontents


%\subjclass{}

\section{Introduction}
%grohs2024stable

The concept of frames, introduced by Duffin and Schaeffer in 1952 \cite{duffin1952class} in the context of nonharmonic Fourier series, has developed into a foundation of modern harmonic analysis and signal processing. Frames provide a robust generalization of orthonormal bases in Hilbert spaces, offering the flexibility of redundancy while retaining the ability to represent vectors faithfully. The creative work of Daubechies, Grossman, and Meyer \cite{daubechies1986painless} significantly advanced frame theory by its use in wavelet and Gabor frame constructions, cementing its role in signal and image processing and sampling theory (\cite{shen2006image}, \cite{aldroubi2001nonuniform}).
A major development in frame theory is the introduction of continuous frames, also known as generalized frames, where the index set extends beyond discrete sets to locally compact spaces equipped with Radon measures. This concept, independently proposed by Kaiser \cite{kaiser1994friendly} and by Ali, Antoine, and Gazeau \cite{ali1993continuous}, extend using of frames, particularly in mathematical physics and quantum mechanics. Gabardo and Han \cite{gabardo2003frames} referred to these structures as frames associated with measurable spaces, while Askari-Hemmat, Dehghan, and Radjabalipour \cite{askari2001generalized} introduced the term generalized frames. Continuous frames are closely related to coherent states, making them a fundamental tool in the study of quantum systems \cite{ali1993continuous}.

For harmonic analysts, continuous frames offer a flexible framework to address complex problems such as phase retrieval and norm retrieval. These problems, central in modern analysis, involve reconstructing information from the magnitudes of frame coefficients. When the underlying measure space is taken as the natural numbers equipped with the counting measure, continuous frames reduce to the classical discrete frames, bridging the gap between discrete and continuous settings. Comprehensive treatments of frame theory and its diverse applications can be found in \cite{ali2000coherent, chui1992introduction, fornasier2005continuous, thirulogasanthar2004continuous,grohs2024stable}.
The study of phase retrieval frames, which enable the reconstruction of signals from the magnitudes of frame coefficients, is particularly significant. This process, essential in various fields such as X-ray crystallography and speech recognition, addresses situations where only magnitude information is available. Norm retrieval focuses on reconstructing the norm of signals from frame coefficients. Together, these properties provide critical insights into the structural and functional aspects of Hilbert spaces. 

This paper studies continuous frames with an emphasis on their phase and norm retrieval properties. We begin by establishing a norm bound for Bessel mappings on measure spaces that satisfy a positive infimum measure condition, ensuring boundedness and providing a basis for further investigations. We also introduce the concept of continuous near-Riesz bases, which generalize formal Riesz bases to the setting of continuous frames. We demonstrate their invariance under invertible operators, highlighting their stability under transformations.
Furthermore, we extend various conditions for phase retrieval in continuous frames setting, including integral inequalities and connections to the $\mu$-complete property. We also address the stability of phase retrieval under perturbations, demonstrating that while specific conditions can destabilize phase retrieval, stability can be ensured under bounded perturbations in finite measure spaces.

In the context of tensor product Hilbert spaces, we investigate tensor product frames, establishing that the phase and norm retrieval properties of these frames are inherently linked to those of their component frames. This result provides a deeper understanding of the interplay between frame properties in product spaces and their implications for practical applications.

The findings presented here expand the theoretical framework of continuous frames, with direct implications for signal processing and quantum mechanics, where phase and norm retrieval are critical for reconstructing signals from incomplete or noisy data. 
This study provides a comprehensive framework for analyzing continuous frames and their retrieval properties, offering new insights and tools for addressing practical challenges in harmonic analysis and beyond.

The paper is divided into five main parts.
It begins with basic concepts of continuous frames and the establishment of norm bounds for Bessel mappings in measure spaces with a positive infimum condition.  
The concept of continuous near-Riesz bases is introduced, alongside their stability under invertible operators. The equivalence of various conditions for phase retrieval is investigated, including integral inequalities and the $\mu$-complete property. Stability analysis of phase retrieval under perturbations is conducted, with specific attention to finite measure spaces. Finally, tensor product frames are studied, demonstrating the connection between the retrieval properties of component frames and their tensor products.

\section{Continuous Frames}
In this section we review the notion of continuous frames, which allow us to analyze and reconstruct elements in a Hilbert space using $(X, \mu)$ as a measure space.
The idea of continuous frames, was introduced by Ali, Antoine, and Gazeau \cite{ali1993continuous} and Kaiser \cite{kaiser1994friendly}. Moreover, we provide a brief discussion on the norm boundeness of continuous frames. 
Through this paper we assume that Hilbert spaces are separable and the measure space $(X, \mu)$ is $\sigma$-finite. Indeed every continuous frame, or more generally a continuous Bessel family is supported on a $\sigma$-finite set. This condition arises from the existence of a continuous frame, ensuring that the synthesis and analysis operators are well-defined. See proposition 2.1 of \cite{bownik2018lyapunov}. Also if $(X,\mu)$ is a $\sigma$-finite measure space, then there exists a continuous tight frame with respect to $(X,\mu)$ \cite{balazs2012multipliers}.
There are several ways to develop a theory of integrals for functions with values in a topological vector space. We shall adopt the "weak" approach,
in which one reduces everything to scalar functions by applying linear functionals.

\begin{definition}
Let $H$ be a Hilbert space and $(X, \mu)$ be a measure space with positive measure $\mu$. The function $F: X \rightarrow H$ is called to be a continuous frame with respect to $(X, \mu)$, whenever,
\begin{itemize}
\item[1.] $x \mapsto F(x)$ is weakly measurable, i.e., for all $f \in H$, the mapping $X \rightarrow \mathbb{C}, x \mapsto\left\langle f, F(x)\right\rangle$ is measurable.\\
\item[2.] There exist constants $A, B>0$ such that
\begin{align}\label{cts_frame}
A\|f\|^2 \leq \int_{X}|\langle f, F(x)\rangle|^2 \mathrm{~d} \mu(x) \leq B\|f \|^2, \quad \text { for all } f \in H.
\end{align}
\end{itemize}
\end{definition}
The constants $A$ and $B$ are called frame bounds. A frame $F$ is said to be tight if we can choose $A=B$; if furthermore, $A=$ $B=1$, then $F$ is called a Parseval frame. 
We recall that, if $F: X \rightarrow H$ is weakly measurable and the upper bound in the inequality (\ref{cts_frame}) holds, then $F$ is said to be a Bessel mapping. 
The analysis operator, denoted by \( T_F \), maps elements of \( H \) to the Hilbert space \( L^2(X, \mu) \) and is defined by  
\[
T_F f(x) = \langle f, F(x) \rangle, \quad \text{for all } f \in H \text{ and } x \in X.
\]  
The synthesis operator, denoted by \( T_F^* \), is the adjoint of the analysis operator. It maps elements of \( L^2(X, \mu) \) to \( H \) and is given by  
\[
T_F^* g = \int_X g(x) F(x) \, d\mu(x), \quad \text{for } g \in L^2(X, \mu).
\]  
The frame operator, denoted by \( S \), is the composition of the synthesis and analysis operators. 
$$
S f = T_F^* T_F f = \int_X \langle f, F(x) \rangle F(x) \, d\mu(x) \quad \text{for} \quad f \in H.
$$
A Bessel mapping $F : X \rightarrow H$ is termed $\mu$-complete if
$$
\operatorname{cspan} \left\{ F(x) \mid x \in X \right\} := \left\{\; \int_X \phi(x) F(x) d\mu(x) \; \mid \; \phi \in L^2(X) \; \right\}  
$$
is dense in $H$, where $\int_X \phi(x) F(x) d\mu(x)$ consider in a weak sense. A mapping \( F: X \to H \) is called a continuous Riesz basis for \( H \) with respect to \((X, \mu)\), if the following holds.  
\begin{itemize}  
\item[1.]  The family \( \{F(x)\}_{x \in X} \) is \(\mu\)-complete.
\item[2.] There exist constants \( A, B > 0 \) (called the Riesz basis bounds) such that for every \( \phi \in L^2(X) \) and every measurable subset \( X_1 \subset X \) with \(\mu(X_1) < +\infty\), the following inequality holds:  

\begin{align*}  
A \left( \int_{X_1} |\phi(x)|^2 \, d\mu(x) \right)^{1/2}  
\leq \left\| \int_{X_1} \phi(x) F(x) \, d\mu(x) \right\|  
\leq B \left( \int_{X_1} |\phi(x)|^2 \, d\mu(x) \right)^{1/2}.  
\end{align*}  
\end{itemize}  
The integral is understood in the weak sense.

The study of continuous frames will pass through the measure theory. We now turn our attention to various types of measures characterized by their atomic structure. An atom is a measurable set that cannot be subdivided into smaller sets of positive measure. This concept is analogous to the idea of an atom in physics as an indivisible unit. More precisely, a measurable set $E$ of positive measure is an atom if for any measurable subset $F$ of $E$ either $\mu(F)=0$ or $\mu(E - F)=0$. A measure is non-atomic or atomless if every set of positive measure can be divided into two sets of positive measures. Such measures are continuous in a sense and do not contain atoms. In fact, in a non-atomic measure space, every measurable set of positive measure can be split into two disjoint measurable sets, each having positive measure. 
A measure is called purely atomic or simply atomic if every measurable set of positive measure contains an atom. In other words, a purely atomic measure is one where every set of positive measure can be partitioned into atoms.
Every measure can be uniquely decomposed as a sum of a purely atomic and a non-atomic measure. See \cite{johnson1970atomic}.

\begin{lemma}\cite{askari2001generalized}\label{atom}
Let $(X, \mu)$ be a measure space, and define $\eta=\inf \{\mu(E):  \; 0<\mu(E)<\infty\}$. The following assertions are true:
\begin{itemize}
\item[(a)] If $\eta=0$, then there exists a sequence of disjoint measurable sets $F_1, F_2, \ldots$ such that $\mu\left(F_n\right)>0$ and $\lim _{n \rightarrow \infty} \mu\left(F_n\right)=0$.
\item[(b)] If $\eta>0$, then every set of positive finite measure is a finite union of disjoint atoms.
\item[(c)] Every measurable function is $\mu$-almost constant on an atom.
\end{itemize}
\end{lemma}

In the setting of a measure space \((X, \mu)\), we introduce the notation \(\eta\) for the purpose of clarity and ease of reference. Specifically, \(\eta\) is defined as the infimum of all non-zero finite measure values of measurable subsets of \(X\). When considering the Lebesgue measure on \( \mathbb{R} \), for any positive \( \epsilon > 0 \), we can construct an interval \( (a, a + \epsilon) \) in \( \mathbb{R} \) whose measure is \( \epsilon \), it follows that for any small \( \epsilon \), there exists a subset of \( \mathbb{R} \) with measure \( \epsilon \). So in this case $\eta = 0$. 

Suppose that $\mu$ is a Radon measure on $X$ such that $\mu(\{x\})=0$ for all $x \in X$, and $A \in \mathcal{B}_X$, the Borel subsets of locally compact Hausdorff space X, satisfies $0<\mu(A)<\infty$. Then for any $\alpha$ such that $0<\alpha<\mu(A)$ there is a Borel set $B \subset A$ such that $\mu(B)=\alpha$. So in this case $\eta = 0$. 

When considering the Haar measure of a locally compact group, for non-discrete group $\eta = 0$. Also for discrete groups, where each singleton set is measurable and typically has a uniform measure (e.g., each point in $\mathbb{Z}$ having measure 1 under counting measure), $\eta$ simply equals the measure of a singleton, assuming the measure is uniformly defined across singletons.

\subsection{Norm boundedness of continuous frame}

The property of norm boundedness plays a fundamental role in ensuring the stability and well-defined nature of operations involving \( F \), where \( F \) is a continuous frame. For example, when \( F \) is employed to reconstruct elements of \( H \) from measurements, norm boundedness prevents the outputs from becoming unbounded, which is essential for the practical reliability and applicability of any reconstruction procedure. It is known that discrete frames are norm bounded, while continuous frames are not so in general. For instance consider the function $\omega: \mathbb{R} \rightarrow \mathbb{R}$ defined by
$$
\omega(x) = \begin{cases} 
      \frac{1}{\sqrt[6]{|x|}}, & \text{if } 0<|x|<1 \\
      \frac{1}{|x|}, & \text{if } |x| \geq 1 \\
      0, & \text{if } x=0 .
   \end{cases}
$$
Let $ H $ be a Hilbert space and $ h$ be a fixed non-zero vector in $H$. Define $ F: \mathbb{R} \rightarrow H $ by $ F(x) = \omega(x)h $. This mapping is weakly (Lebesgue) measurable and a continuous Bessel mapping. However, $\|F(x)\|$ is unbounded. See \cite{rahimi2017construction}. Note that in this example $\eta = 0$.

\begin{lemma}\label{bessel_bound}
Let $H$ be a Hilbert space and $(X,\mu)$ be a $\sigma$-finite positive measure space, such that 
\begin{equation}\label{eta_formula}
0 < \eta=\inf \{\mu(E): \; 0<\mu(E)<\infty\}.
\end{equation}
Then every Bessel mapping $F: X \rightarrow H$ with respect to $(X, \mu)$ is norm bounded.
\end{lemma}

\begin{proof}
Since $F$ is a Bessel mapping with respect to $(X, \mu)$, there exists a constant $B > 0$ such that for all $f \in H$, the upper inequality of (\ref{cts_frame}) holds. Let $ X = \cup_{i = 1}^{\infty} X_i$ such that $\mu(X_i) < \infty$. Now, by Lemma \ref{atom} (b), there exists a finite family of disjoint atoms $\{E_{ij}\}_{i=1}^n$ so that $X_j = \cup_{i=1}^n E_{ij}$. Let us consider atom $E_{i_k} \subseteq X$. We know that $\mu(E_{i_k}) > 0$. Choose any $f \in H$ and consider the integral over the single atom $E_{i_k}$. By Lemma \ref{atom} the function $F$ is almost everywhere constant on $E_{i_k}$, so for any $x_0 \in E_{i_k}$
\begin{align*}
|\langle f, F(x_0) \rangle|^2 \mu(E_{i_k}) = \int_{E_{i_k}} |\langle f, F(x) \rangle|^2 d\mu(x) 
\leq B \|f\|^2.
\end{align*}
The last inequality is obtained by the Bessel property of $F$. Because $\mu(E_{i_k}) \geq \eta > 0$, we can rearrange this inequality to obtain a bound for $\langle f, F(x_0) \rangle$ as follows 
$$
|\langle f, F(x_0) \rangle|^2 \leq \frac{B}{\mu(E_{i_k})} \|f\|^2 \leq \frac{B}{\eta} \|f\|^2.
$$
Inequality holds true independent of the choice of $i$ and $k$, where $x_0 \in E_{ik}$ and $f \in H$.  Therefore, $\|F(x)\|$ is bounded above by $\sqrt{B/\eta}$, and this upper bound is valid for all $x \in X$. Hence, $F$ is norm-bounded above.
\end{proof}

Given the conditions in Lemma \ref{bessel_bound}, any continuous frame $F: X \rightarrow H$ with respect to the measure space $(X, \mu)$ is also norm bounded. Continuous Bessel mappings can be either bounded or unbounded, depending on the properties of the measure space. In the following, we present some examples that show this concept. The examples provided do not only show the application of Lemma \ref{bessel_bound}, but also highlight its limitations and the conditions under which it applies.

\begin{example}
Consider the following two examples with  $X=\mathbb{R}:$
\begin{itemize}
\item[1-] Let $H=\ell^2(\mathbb{Z})$, let $\left\{e_k\right\}_{k \in \mathbb{Z}}$ be its standard orthonormal basis, and equip $X=\mathbb{R}$ with a purely atomic measure, whose atoms are the intervals $[n, n+1), n \in \mathbb{Z}$, each with measure 1. Define $\left\{F(x) \mid x \in \mathbb{R} \right\} \subset H$ by $F(x) =e_{\lfloor x\rfloor}, x \in \mathbb{R}$.
For any \(x \in \ell^2(\mathbb{Z})\),
    \[
    \|x\|^2 = \sum_{k \in \mathbb{Z}} |x_k|^2 = \sum_{k \in \mathbb{Z}} \int_{[k, k+1)} |\langle x, e_k \rangle|^2 \, d\mu = \int_{\mathbb{R}} |\langle x, F(x) \rangle|^2 \, d\mu.
    \]
So the family $\left\{F(x) \mid x \in \mathbb{R}\right\}$ is a Parseval frame.

\item[2-] Let $H=L^2([0,1])$, and let $X=\mathbb{R}$. Fix $a \in \mathbb{R}$ and define $\mu_X=\sum_{n \in \mathbb{Z}} \delta_{n+a}$, where $\delta_x$ denotes the Dirac measure at $x \in \mathbb{R}$. Define $\left\{F(x) \mid x \in \mathbb{R}\right\} \subset H$ by $F(x)y= e^{2 \pi i x y}$. 
In this case, the smallest measurable subsets that have a non-zero measure are the singletons containing each point \( n + a \), and each such singleton has a measure of 1 due to the Dirac delta measure at that point. Also, there are no smaller subsets with a positive measure. Therefore, \( \eta = 1 \). This example illustrates how the choice of measure on a space \( X \) acts on properties like \( \eta \), especially in cases where the measure is concentrated at specific points.
For each $x \in \mathbb{R}$, define a function $F(x)$ by $F(x)y = e^{2 \pi i x y}$. This function maps $y \in [0,1]$ to the complex exponential $e^{2 \pi i x y}$. Each $F(x)y$ is an element of $L^2([0,1])$. Functions of the form $e^{2 \pi i n y}$ for integer $n$ are known to form an orthonormal basis for $L^2([0,1])$. Given $\mu_X = \sum_{n \in \mathbb{Z}} \delta_{n+a}$, we rewrite the integral as a sum,
$$
\int_{\mathbb{R}} |\langle f, F(x) \rangle|^2 \, d\mu_X = \sum_{n \in \mathbb{Z}} |\langle f, F(n+a) \rangle|^2 =\|f\|^2 .
$$
So the family $\left\{F(x) \mid x \in \mathbb{R}\right\}$ is a Parseval frame for $L^2([0,1])$.

\end{itemize}

\end{example}

\section{Phase Retrieval}

Two essential properties in the study of frames are completeness and the ability to perform phase retrieval. These properties describe the capacity of a frame to represent all elements of a Hilbert space and to reconstruct a signal from the frame elements and the magnitudes of its frame coefficients.
\begin{definition}
A continuous frame $F: X \rightarrow H$ does phase retrieval (norm retrieval) if all $f, g \in H$ which satisfy
$$
|\langle f, F(x) \rangle| = |\langle g, F(x) \rangle| \quad \text{for $\mu$-almost all } x \in X,
$$
implies that $f = e^{i\theta}g$ for some $\theta \in \mathbb{R}$ ($\|f\| = \|g\|$).
\end{definition}
Consider the nonlinear mapping $\mathcal{A}_{F}: H \rightarrow L^2(X)$, given by
\begin{equation}\label{A_f_non_linera_map}
(\mathcal{A}_{F}(f))(x) =\left|\left\langle f, F(x)\right\rangle\right| .
\end{equation}
It is said that $F$ does phase retrieval if and only if $\mathcal{A}_{F}$ is injective on $H / \sim$, containing the equivalent classes of $H$ so that $f \sim g$ if $f=\alpha g$ for some scalar $\alpha$ with $|\alpha|=1$.

In the other form a continuous frame does phase retrieval if the magnitudes of the frame coefficients determine an element in the Hilbert space up to a unimodular constant. This property is crucial in applications where phase information may be lost or inaccessible. Clearly, any continuous frame which yields phase retrieval necessarily yields norm retrieval, but the converse does not hold in general.

It is known that $ \mathcal{A}_{F}$ does not take lower Lipschitz bound. However, it takes a finite upper Lipschitz bound. Indeed let $H$ be a Hilbert space and let $(X, \mu)$ be a positive measure space. Let $F: X \rightarrow H$ be a continuous frame for $H$, with frame bounds $0<A \leq B<\infty$. Then for every $f, g \in H$ we have
$$
\left\|\mathcal{A}_{F}(f)-\mathcal{A}_{F}(g)\right\| \leq B^{1 / 2} \inf _{|\alpha|=1}\|f-\alpha g\| .
$$
First note that
$$
||\left\langle f, F(x) \right\rangle|-|\left\langle g, F(x) \right\rangle \| \leq\left|\left\langle f, F(x)\right\rangle-\left\langle g, F(x)\right\rangle\right|
$$
by the reverse triangle inequality. This means that
$$
\left\|\mathcal{A}_{F}(f)-\mathcal{A}_{F}(g)\right\| \leq\left\|T_{F} f-T_{F} g\right\| \leq B^{1 / 2}\|f-g\|,
$$
where $T_{F}$ is the analysis operator of $F$.
Since $\mathcal{A}_{F}(\alpha g)=\mathcal{A}_F(g)$ for any unimodular scalar $\alpha$, we obtain
$$
\left\|\mathcal{A}_{F}(f)-\mathcal{A}_{F}(g)\right\|=\inf _{|\alpha|=1}\left\|\mathcal{A}_{F}(f)-\mathcal{A}_{F}(\alpha g)\right\| \leq B^{1 / 2} \inf _{|\alpha|=1}\|f-\alpha g\| .
$$

The following definition presents an analogical version of the $\mu$-complement property in the discrete case \cite{balan2006signal}. 

\begin{definition}
We say a continuous frame $F: X \rightarrow H$  has the $\mu$-complement property if for every measurable subset $S \subseteq X$ we have
$$
\operatorname{\overline{span}}\left\{F(x)  \mid  x \in S\right\}=H \quad \text { or } \quad \operatorname{\overline{span}}\left\{F(x)  \mid x \notin S\right\}= H.
$$

\end{definition}

Suppose \( F \) is a continuous frame for a Hilbert space \( H \) with respect to the measure space \( (X, \mu) \) and with bounds \( A \) and \( B \). If \( U: H \rightarrow K \) is a bounded, surjective operator where \( K \) is another Hilbert space, then \( UF \) constitutes a continuous frame for \( K \) with respect to \( (X, \mu) \). The frame bounds for \( UF \) are given by \( A\|U^\dagger\|^{-2} \) and \( B\|U\|^2 \), where \( U^\dagger \) is the pseudo-inverse of \( U \). See Corollary 2.14 \cite{rahimi2006continuous}. Moreover for a continuous frame $F$ of $H$, it is shown that $UF$ is a continuous frame for $K$ if and only if $U$ is surjective. See Theorem 2.6. \cite{faroughi2012c}.

Let us now define a specific type of continuous frame known as a continuous near-Riesz basis.

\begin{definition}\label{near_def}
A continuous frame $F: X \rightarrow H$ for a Hilbert space $H$ is called a continuous near-Riesz basis if there exists a finite positive measurable subset $X_1 \subset X$ and $0 < \mu(X_1) < \infty$, such that the set $\{F(x) \mid x\in X \backslash X_1 \}$ forms a Riesz basis. And, if $\mu(X) = \infty$, then for all $Y \subset X$ such that $\mu(Y) < \infty$, deduces that $F |_Y $ is not $\mu-$complete.
\end{definition}

Based on the above definition, we can say that every continuous Riesz basis is naturally a continuous near-Riesz basis. Also it is worthwhile to note that continuous near-Riesz basis is preserved by the topological isomorphism. In fact we have:

\begin{proposition}
Let \((X,\mu)\) be a measure space, and \( F: X \rightarrow H \) be a continuous near-Riesz basis for a Hilbert space \(H\), and \( U: H \rightarrow H \) be an  invertible operator, then $UF$, is also a continuous near-Riesz basis.
\end{proposition}
\begin{proof}
First we note that, obviously $UF$ is weakly measurable. Moreover, since $F$ is a continuous near-Riesz basis so there exists a positive measurable subset $X_1 \subset X$, where $0 < \mu(X_1) < \infty$, such that the set $\{F(x) \mid x\in X \backslash X_1 \}$ forms a Riesz basis. So $\{U(F(x)) \mid x\in X \backslash X_1 \}$ forms a Riesz basis as well. Moreover, if $\mu(X) = \infty$, then for all $Y \subset X$ such that $\mu(Y) < \infty$, deduces that $F |_Y $ is not $\mu-$complete, so there exist $0 \neq f \in H$ such that $\langle f, F(x)\rangle = 0$ for almost all $x \in Y$. Since $U$ is invertible we have
$$
\langle (U^{-1})^* f, UF(x)\rangle = \langle f, F(x)\rangle = 0,
$$
for almost all $x \in Y$. Hence $UF \mid_Y$ is not $\mu$-complete. This completes the proof.
\end{proof}

The following proposition provides an equivalent condition for the classifying a Bessel mapping $F$ to be $\mu$-complete. 
\begin{proposition}\label{Pro_1_2}\cite{arefijamaal2013new}
Let $F \in L^2(X, H)$ be a Bessel mapping. The following are equivalent:
\begin{itemize}
\item[(a)] $F$ is $\mu$-complete.
\item[(b)] If $f \in H$ so that $\langle f, F(x)\rangle = 0$ for almost all $x \in X$, then $f = 0$. See \cite{arefijamaal2013new}.
\end{itemize}
\end{proposition}

The following theorem establishes the equivalence of several mathematical conditions that characterize when a continuous frame in a real Hilbert space can perform phase retrieval. The proof of this theorem is based on Lemma \ref{bessel_bound}, which asserts that every Bessel mapping in a Hilbert space, defined on a \(\sigma\)-finite positive measure space, is norm bounded.

\begin{theorem}\label{phase_equvalent_cond}
Let \( F : X \to H \) be a continuous frame for a real Hilbert space \( H \) with respect to a measure space \( (X, \mu) \). Then the following are equivalent:
\begin{itemize}
\item[(a)] $F$ has $\mu$-complete property.
\item[(b)] $F$ is phase retrieval.
\item[(c)]For any two non-zero vectors \( f, g \in H \), we have:
\[
\int_X \left|\langle f, F(x) \rangle\right|^2 \left|\langle g, F(x) \rangle\right|^2 d\mu(x) > 0.
\]
\item[(d)] There is a positive real constant $\alpha > 0$ so that for all $f, g \in H$,
$$
\begin{aligned}
\int_{X}\left|\left\langle f, F(x)\right\rangle\right|^2\left|\left\langle g, F(x)\right\rangle\right|^2 d\mu(x) \geq \alpha\|f\|^2\|g\|^2
\end{aligned}.
$$
\item[(e)] There is a positive real constant $\alpha >0$ so that for all $f \in H$,
\begin{align}\label{2_6 in theorem}
R(f):=\int_{X} \left|\left\langle f, F(x)\right\rangle\right|^2\left\langle\cdot, F(x)\right\rangle F(x) d\mu(x) \geq \alpha I,
\end{align}
where the inequality is in the sense of quadratic forms.

\end{itemize}
\end{theorem}
\begin{proof}
$(a) \Rightarrow (b)$ 
Suppose $F$ does not do phase retrieval, so we can find some nonzero vectors $f, g \in H$ so that $|\langle f, F(x) \rangle|=|\langle g, F(x) \rangle|$ for $\mu$ almost every where $x \in X$, but $f \neq \pm g$. Since $H$ is a real Hilbert space this means that $\langle f, F(x) \rangle= \pm\langle g , F(x) \rangle$ for $\mu$ almost every where. Thus, let $S=\{x \in X :\langle f , F(x) \rangle=\langle g , F(x) \rangle\}$. Then $S$ is a measurable subset of $X$, moreover $f - g \neq 0$ but $\langle f - g, F(x) \rangle=0$, for $x \in S$. Thus $\operatorname{\overline{cspan}}\{F(x) \mid x \in S \} \neq H$, and similarly $f + g \neq 0$ but $\langle f + g , F(x) \rangle=0$ for every $x \notin S$, so $\operatorname{\overline{cspan}}\{ F(x) \mid x \notin S \} \neq H$, which means that $F$ does not have the $\mu$-complement property.\\
$(b) \Rightarrow (a)$
Assume that $F: X \rightarrow H$ is a continuous frame which does phase retrieval. We will show that $F$ has the $\mu$-complement property. Suppose by contradiction that $F$ does not have the $\mu$-complement property. This means there exists a measurable subset $S \subseteq X$ such that neither $\operatorname{\overline{cspan}}\{F(x)  \mid  x \in S\}$ nor $\operatorname{\overline{cspan}}\{F(x)  \mid  x \notin S\}$ is equal to $H$. By Proposition \ref{Pro_1_2}, we can deduce that there exist some nonzero vectors $f,g \in H$ such that $\langle f, F(x)\rangle = 0$ for almost all $x \in S$ and $\langle g, F(x)\rangle = 0$ for almost all $x \notin S$. Since $F$ is a frame we know that $f \neq \lambda g$ for any scalar $\lambda$, so in particular, $f+g \neq 0$ and $f-g \neq 0$. It now follows that $|\langle f+g, F(x)\rangle|=|\langle f-g, F(x)\rangle|$ for all $x \in X$ but $f+g \neq \lambda(f-g)$ for any scalar $\lambda$, that is a contradiction.\\
$(b) \Rightarrow (c)$ 
Suppose by contradiction that there exist non-zero vectors \( f, g \in H \) such that:
\[
\int_X \left|\langle f, F(x) \rangle\right|^2 \left|\langle g, F(x) \rangle\right|^2 d\mu(x) = 0.
\]
This implies that for almost every \( x \in X \), \( \langle f, F(x) \rangle \langle g, F(x) \rangle = 0 \). Define:
\[
X_f = \{ x \in X : \langle f, F(x) \rangle = 0 \} \quad \text{and} \quad X^c_f = X \setminus X_f.
\]
Since \( \langle f, F(x) \rangle = 0 \) for all \( x \in X_f \), the set \( F(X_f) \) of frame elements indexed by \( X_f \) cannot span \( H \). Similarly, \( g \) must be orthogonal to all elements of \( F(X^c_f) \). Therefore, \( F(X^c_f) \) does not span \( H \), as well. By  equivalence of $(a)$ and $(b)$, this contradicts the assumption of phase retrieval.\\
$(c) \Rightarrow (d) \Rightarrow(e) \Rightarrow(b)$  can be shown by an analogous approch to  \cite{balan2012reconstruction}.
\end{proof}

It is worth noting that part $(b) \Rightarrow (a)$ also holds true when considering a Hilbert space over complex numbers as well.

\begin{corollary}
Let \( F : X \to H \) be a continuous near-Riesz basis for a Hilbert space \( H \) with respect to a measure space \( (X, \mu) \).  Then \( F \) is not phase retrieval.
\end{corollary}
\begin{proof}
Suppose that \( F \) is a continuous near-Riesz basis, so there exists a subset \( X_1 \subset X \) with \( 0 < \mu(X_1) < \infty \) such that the set \( \{ F(x) \mid x \in X \setminus X_1 \} \) forms a Riesz basis for \( H \). Let $Y \subset X - X_1$ such that $ 0 < \mu(Y) < \infty$ then $F |_{X_1 \cup Y}$ is not $\mu$-complete and also $F|_{X - (X_1 \cup Y) }$ is not $\mu$-complete. So by the Theorem \ref{phase_equvalent_cond} it is not a phase retrieval.

\end{proof}

The stability of phase retrieval property under perturbations of the frame set is given in the next theorem.

\begin{theorem}\label{non1}
Let $H$ be an infinite dimensional Hilbert space and $(X, \mu)$ be a positive $\sigma$-finite measure space  with $\mu(X) = \infty$. Also if $F : X \rightarrow H$ is a phase retrieval continuous frame with bounds $A,B$. And $F$ on finite measurable subset be not $\mu$-complete. Then for any $A > \epsilon > 0$, there exist $G: X \rightarrow H$ such that $G$ is not phase retrieval, and
\begin{equation}\label{Stable_1}
\int_{X}\left\| F(x) - G(x)\right\|^2 d\mu(x)<\epsilon.
\end{equation}

\end{theorem}

\begin{proof}
We assume, without loss of generality, $X_i \subseteq X_{i+1}$. Suppose $F: X \rightarrow H$, is a continuous frame with respect to $(X, \mu)$. It is enough to show that, for every given $\epsilon>0$, there exists another frame $G$ that does not yield phase retrieval and satisfies (\ref{Stable_1}).
Given $\epsilon > 0$, it is worth noting that $F$ being a continuous Bessel mapping implies that the mapping $X \rightarrow \mathbb{C}$ defined by $x \mapsto \left\langle f, F(x) \right\rangle$ is measurable for all $f \in H$. Moreover, there exist $B > 0$ so that
$$
\int_{X} \left| \left\langle f, F(x) \right\rangle \right|^2 \mathrm{d}\mu(x) \leq B \|f\|^2,
$$
for all $f \in H$. According to \cite{folland1999real} and without loss of generality, we can deduce there exists $i_0 \in \mathbb{N}$ such that
\begin{equation}\label{label_1}
\int_{X-X_{i_0}} \left| \left\langle f, F(x) \right\rangle \right|^2 \mathrm{d}\mu(x) < \epsilon \quad \text{for all } f \in H.
\end{equation}
Let $x_1 \in X_{i_0}$ so that $F(x_1) \neq 0$ and put $\frac{F(x_1)}{\| F(x_1) \|}= e_1 \in H$. Define the map $G: X \rightarrow H$ by
$$
G(x) = 
\begin{cases} 
F(x) & \text{if } x \in X_{i_0}, \\
F(x) -\overline{\langle e_1, F(x) \rangle } e_{1} & \text{if } x \in X-X_{i_0}.
\end{cases}
$$
Then $G$ is close to $F$ in terms of 
$$
\int_X\|F(x)-G(x)\|^2 d \mu(x)=\int_{X-X_{i_0}}\left|\left\langle e_1, F(x)\right\rangle\right|^2 d \mu(x)<\epsilon,
$$
and 
\begin{align*}
\int_X | \left\langle F(x) - G(x) , f \right\rangle |^2 d\mu(x) &= | \left\langle e_1 , f \right\rangle| ^2 \int_{X-X_{i_0}} | \left\langle e_1 , F(x) \right\rangle |^2 d\mu(x) \\
&\leq \epsilon \|f\|^2, \quad (\text{for all } f \in H),
\end{align*}
where the last inequality holds according to (\ref{label_1}). So by the Corollary 4.3. of \cite{rahimi2006continuous}, $G$ is a continuous frame. To demonstrate that $G$ does not yield phase retrieval, choose vector  $e_2 \neq 0$ so that 
$
e_2 \in \{ F(x) \mid x \in X_{i_0}\}^\perp
$. Then $e_1 \perp e_2$.
Put $f=e_1 + 2e_2$ and $g=e_1 - 2e_2$. Then for almost every $x \in X_{i_0}$,
\begin{align*}
|\langle f, G(x) \rangle| &= |\langle e_1 + 2e_2, F(x)\rangle| \\
&=  |\langle e_1, F(x)\rangle| \\
&= |\langle e_1 - 2e_2, F(x)\rangle| \\
&=  |\langle g, G(x) \rangle| .
\end{align*}

Also, for almost every $x \in X - X_{i_0}$ we get
\begin{align*}
|\langle f, G(x) \rangle| &= |\langle e_1 + 2e_2, F(x) - \overline{\langle  e_1 , F(x) \rangle } e_{1} \rangle| \\
&= |\langle e_1 , F(x)  \rangle  
- \langle e_1,F(x) \rangle +
\langle 2e_2 , F(x)  \rangle 
| = |\langle 2e_2, F(x)\rangle|.
\end{align*}
Moreover, similarly
$$
|\langle g, G(x) \rangle| = |\langle e_1 - 2e_2, F(x) - \overline{\langle e_1,F(x) \rangle}  e_{1} \rangle| = |\langle 2e_2, F(x)\rangle|,
$$
for almost every $x \in X - X_{i_0}$.

So $|\langle f, G(x) \rangle| = |\langle g, G(x) \rangle| $ for almost all $x \in X$, but clearly $e_1$ and $e_2$ are linearly independent so $f \neq e^{i \theta} g \text { for any } \theta \in \mathbb{R}$. Hence, we have shown that $G$ does not do phase retrieval. This completes the proof.

\end{proof}

The following theorem states that if the perturbation \( G \) of the frame \( F \) is sufficiently small (i.e., \( \operatorname{sup}_{x \in X} \| F(x) - G(x) \| < \lambda \) for some \( \lambda > 0 \)), then \( G \) also forms a phase retrievable frame for \( H \). This statement implies a robustness in the property of phase retrievability under small changes to the frame elements. The following result can be shown in a similar approach to Theorem 2.5 in \cite{fard2016norm}. So its proof is deleted.

\begin{theorem}\label{Th_finite}
Let $(X,\mu)$ be a finite measure space such that 
\begin{equation}\label{eta_formula}
0 < \eta=\inf \{\mu(E): E \text{ measurable}, \; 0<\mu(E)<\infty\}.
\end{equation}
If $F : X \rightarrow H$ is a phase retrievable continuous frame for real hilbert space $H$ with frame bounds $A$ and $B$. Then there exists $\lambda>0$ such that any map $G : X \rightarrow H$, satisfying  $\operatorname{sup}_{x \in X} \| F(x) - G(x) \|<\lambda$, is also a phase retrievable frame for $H$.
\end{theorem}

\begin{corollary}
Let $G$ be a compact group with a Haar measure $\mu$ and $\eta > 0$. Suppose $F_1 : G \rightarrow H$ is a phase retrievable continuous frame for a real Hilbert space $H$ with frame bounds $A, B$. Then, there exists $\lambda > 0$ such that any map $F_2 : G \rightarrow H$ satisfying $\operatorname{sup}_{x \in G} \| F_1(x) - F_2(x) \| < \lambda$ is also a phase retrievable frame for $H$.
\end{corollary}

\section{Norm Retrieval}

In the study of continuous frames, understanding the orthogonality properties between subspaces generated by partitions of the frame set offers valuable insights into the structure and functionality of these frames. The following theorem establishes an important orthogonality property for continuous frames. This result not only highlights a fundamental property of norm retrievable continuous frames, but also generalizes a corresponding result for discrete frames presented in \cite{fard2016norm}. Let us now state and prove this significant orthogonality property.

\begin{theorem}\label{norm_cts1}
A continuous frame \( F: X \to H \) over a real Hilbert space \( H \) is norm retrievable if and only if for any measurable subset \( \Omega \subseteq X \),  we have the following orthogonality condition
\[
\operatorname{span}\{F(x) \mid x \in \Omega \}^\perp \perp \operatorname{span}\{F(x) \mid x \in \Omega^c\}^\perp.
\]
\end{theorem}

\begin{proof}
Suppose that $F$ is norm retrieval and $ f \in \operatorname{span}\{F(x) \mid x \in \Omega\}^\perp$, $g \in \operatorname{span}\{F(x) \mid x \in \Omega^c\}^\perp$, then 
$\langle f , F(x) \rangle = 0$ for almost all $x \in \Omega$ and $\langle g , F(x) \rangle = 0$ for almost all $x \in \Omega^c$. Then 
$$
\langle f+g , F(x) \rangle = - \langle f-g , F(x) \rangle \quad \text{ for $\mu$ almost all } x \in \Omega,
$$
moreover 
$$
\langle f+g , F(x) \rangle =  \langle f-g , F(x) \rangle \quad \text{ for $\mu$ almost all } x \in \Omega^c.
$$
Hence $ | \langle f+g , F(x) \rangle | = | \langle f-g , F(x) \rangle | $ for almost all $x \in X$ and consequently $\| f + g \| = \|f - g\|$, so $\langle f , g \rangle = 0$, as required.

Conversely let $f, g \in H$ be such that for almost every $x \in X$, 
$$
|\langle f, F(x) \rangle| = |\langle g, F(x) \rangle|.
$$
We must show that $\|f\| = \|g\|$. Define the set
$$
\Omega = \{x \in X : \langle f, F(x) \rangle = \langle g, F(x) \rangle\},
$$
then
$$
\Omega^c = \{x \in X : \langle f, F(x) \rangle = -\langle g, F(x) \rangle\}.
$$
Consider the vectors $f + g$ and $f - g$. 
For $x \in \Omega^c$,
$$
\langle f + g, F(x) \rangle = \langle f, F(x) \rangle - \langle g, F(x) \rangle = 0.
$$
Hence, $f + g \in \operatorname{span}\{F(x) \mid x \in \Omega^c\}^\perp$.

Similarly, for $x \in \Omega$,
$$
\langle f - g, F(x) \rangle = \langle f, F(x) \rangle - \langle g, F(x) \rangle = 0.
$$
Thus, $f - g \in \operatorname{span}\{F(x) \mid x \in \Omega\}^\perp$. Given the orthogonality condition, these spans are orthogonal complements, so
$$
\langle f+g, f-g \rangle = 0.
$$
Expanding this, we get
$$
\|f\|^2 - \|g\|^2 = 0.
$$
This demonstrates that $\|f\| = \|g\|$, establishing that $F$ is a norm retrievable continuous frame.

\end{proof}

\begin{remark}
By Theorem \ref{norm_cts1} orthogonal Riesz basis over real Hilbert space are norm retrieval.
\end{remark}

The stability of frame properties under perturbations is a critical aspect that affects their practical applications in signal processing, communications, and data analysis. The following corollary studies the implications of such perturbations on a norm retrievable, but not phase retrievable, continuous frame defined over a real Hilbert space. In this study, we illustrate the specific conditions under which a norm-retrievable continuous frame loses its ability to maintain norm retrieval when subjected to  perturbations.
\begin{corollary}
Suppose that \( F: X \rightarrow H \) is a norm retrievable, but not phase retrievable continuous frame over the real Hilbert space with bounds $A,B$ and with respect to a $\sigma$-finite measure space \( (X, \mu) \). Then \( F \) does not preserve the norm retrievable property under perturbations.
\end{corollary}
\begin{proof}
Applying the assumption and Theorem \ref{phase_equvalent_cond} there exists a measurable subset $S \subset X$ such that 
$
\operatorname{\overline{cspan}}\left\{F(x)  \mid  x \in S\right\} \neq H$ and $\operatorname{\overline{cspan}}\left\{F(x)  \mid x \notin S\right\} \neq H.
$
This implies the existence of nonzero vectors $f \in \operatorname{\overline{cspan}}\left\{F(x) \mid x \in S \right\}^{\perp}$ and $g \in \operatorname{\overline{cspan}}\left\{F(x) \mid x \neq S \right\}^{\perp}$. Using the norm retrievability of $F$ we have $\langle f, g\rangle=0$ by Theorem \ref{norm_cts1}. Define $F' : X \rightarrow H$ by \( F'(x) = F(x) - \epsilon \delta(x) \), where $\epsilon < 2 \sqrt{A}$ and \( \delta: X \rightarrow H \) is defined as follows
$$
\delta(x)=\left\{\begin{array}{cl}
\frac{\left\langle F(x), g\right\rangle f}{2 \sqrt{B}\|f\|\|g\|} & \text { if } x \in S \\
0 & \text { if } x \in S^c 
\end{array}\right. .
$$
Then $\left\|\delta(x) \right\|<1$ for all $x \in X$.
Then $F'$ is a continuous frame. Indeed 
\begin{align*}
\int_X\left|\left\langle h, F'(x)-F(x)\right\rangle\right|^2 d x &= \int_S \epsilon^2 \left|\left \langle h, \delta(x) \right\rangle\right|^2 d x \\
&\leq \frac{\epsilon^2 \|h\|^2}{4B \|g\|^2} \int_S \left|\left\langle g, F(x)\right\rangle\right|^2 d x \\
&\leq \frac{\epsilon^2}{4} \|h\|^2 .
\end{align*}
So $F'$ is a continuous frame. Moreover, $F^{\prime}$ is not a norm retrievable frame. Indeed, 
\begin{align*}
&\left\langle\frac{2 \sqrt{B}\|g\| f}{\|f\|} +\epsilon g, F^{\prime}(x) \right\rangle = \left\langle\frac{2 \sqrt{B}\|g\| f}{\|f\|}, F^{\prime}(x) \right\rangle + \left\langle \epsilon g, F^{\prime}(x) \right\rangle \\
&= \left\langle\frac{2 \sqrt{B}\|g\| f}{\|f\|}, F(x) - \epsilon \delta(x) \right\rangle + \left\langle \epsilon g, F(x) - \epsilon \delta(x) \right\rangle \\
&=0 - \epsilon \langle F(x),g \rangle + \langle \epsilon g , F(x) \rangle + 0 = 0,
\end{align*}
 for all $x \in S$ and $\left\langle g, F^{\prime}(x)\right\rangle=0$ for all $x \in S^c$. Also we have
$$
\left\langle\frac{2 \sqrt{B}\|g\| f}{\|f\|} + \epsilon g, g\right\rangle=\epsilon\|g\|^2>0.
$$
Consequently, $\operatorname{\overline{cspan}}\left\{F^{\prime}(x) \mid x \in S\right\}^{\perp}$ is not orthogonal to $\operatorname{\overline{cspan}}\left\{F^{\prime}(x) \mid x \in S^c \right\}^{\perp}$, and theorem \ref{norm_cts1} established that $F^\prime$ is not a norm retrievable frame.
\end{proof}

\section{Tensor product Hilbert spaces}

The tensor product $H_1 \otimes H_2$ of the Hilbert spaces $H_1$ and $H_2$ is the set of all antilinear maps $T: H_2 \longrightarrow$ $H_1$ such that $\sum_j\left\|T \alpha_j \right\|^2<\infty$ for an orthonormal basis $\left\{\alpha_j\right\}$ for $H_2$. Also $H_1 \otimes H_2$ is a Hilbert space with the norm
$$
\|T\|^2=\sum_j \| T \alpha_j \|^2,
$$
and the associated inner product
$$
\left\langle T_1, T_2\right\rangle=\sum_j\left\langle T_1 \alpha_j, T_2 \alpha_j\right\rangle,
$$
where $\left\{\alpha_j\right\}$ is any orthonormal basis for $H_2$ and $T_1, T_2$ are the antilinear maps from $H_2$ onto $H_1$. For any vector $f \in H_1$ and $g \in H_2$, the mapping defined by
$$
\left(f \otimes g\right)(t)=\left\langle g, t\right\rangle f, \quad (t \in H_2)
$$
belongs to $H_1 \otimes H_2$.
Also the inner product
\begin{equation}\label{Tensor}
\left\langle f_1 \otimes g_1, f_2 \otimes g_2\right\rangle_{\otimes}=\left\langle f_1, f_2\right\rangle_{H_1}\left\langle g_1, g_2\right\rangle_{H_2}, \quad f_1, f_2 \in H_1, g_1, g_2 \in H_2,
\end{equation}
makes it into a Hilbert space. See \cite{folland2016course}.

Let $H$ be the tensor product $H=H_1 \otimes H_2$ of separable complex Hilbert spaces, and $(X, \mu)=\left(X_1 \times X_2, \mu_1 \otimes \mu_2\right)$ be the product of measure spaces with $\sigma$-finite positive measures $\mu_1, \mu_2$. The mapping $F: X \rightarrow H$ is called a continuous frame for the tensor product Hilbert space $H$ with respect to $(X, \mu)$, if\\
(1) $F$ is weakly-measurable, i.e., for all $f \in H$,
$$
x=\left(x_1, x_2\right) \rightarrow\langle f, F(x)\rangle
$$
is a measurable function on $X$.\\
(2) There exist constants  $0<A\leq B<\infty$ such that
$$
A\|f\|^2 \leqslant \int_X|\langle f, F(x)\rangle|^2 \mathrm{~d} \mu(x) \leqslant B\|f\|^2, \quad \forall f \in H .
$$

Let $H_1$ and $H_2$ be separable Hilbert spaces, $H=H_1 \otimes H_2$, and let $(X, \mu)=$ $\left(X_1 \times X_2, \mu_1 \otimes \mu_2\right)$ be the product of measure spaces with $\sigma$-finite positive measures $\mu_1, \mu_2$. The mapping $F=F_1 \otimes F_2: X \rightarrow H$ is a continuous frame for $H$ with respect to $(X, \mu)$ if and only if $F_1$ is a continuous frame for $H_1$ with respect to $\left(X_1, \mu_1\right)$, and $F_2$ is a continuous frame for $H_2$ with respect to $\left(X_2, \mu_2\right)$. Tensor product of continuous frames was introduced in \cite{balazs2022continuous}.
\begin{definition}
Let $H$ be the tensor product $H=H_1 \otimes H_2$ of separable complex Hilbert spaces, and $(X, \mu)=\left(X_1 \times X_2, \mu_1 \otimes \mu_2\right)$ be the product of measure spaces with $\sigma$-finite positive measures $\mu_1, \mu_2$. A continuous frame \( F: X \to H \) with respect to \( (X, \mu) \) is called phase retrieval if for all $T_1, T_2 \in H_1 \otimes H_2$ satisfying 
$$
\left|\left\langle T_1, F(x)\right\rangle\right|=\left|\left\langle T_2, F(x)\right\rangle\right| \quad \text{for } \mu\text{-almost all } x \in X,
$$ 
it follows that \( T_1 = e^{i\theta}  T_2 \) for some \( \theta \in \mathbb{R} \).
\end{definition}

In the next result, we investigate phase retrieval property of tensor product of continuous frames. This result extends Theorem 3.3 \cite{dowerah2023phase} of discrete case.

\begin{theorem}\label{Th_tensor_phase}
Let $H_1$ and $H_2$ be separable Hilbert spaces, $H=H_1 \otimes H_2$, and let $(X, \mu)=$ $\left(X_1 \times X_2, \mu_1 \otimes \mu_2\right)$ be the product of measure spaces with $\sigma$-finite positive measures $\mu_1, \mu_2$. The mapping $F=F_1 \otimes F_2: X \rightarrow H$ is a phase retrieval continuous frame for $H$ with respect to $(X, \mu)$ if and only if $F_1$ is a phase retrieval continuous frame for $H_1$ with respect to $\left(X_1, \mu_1\right)$, and $F_2$ is a phase retrieval continuous frame for $H_2$ with respect to $\left(X_2, \mu_2\right)$.
\end{theorem}
\begin{proof}

Fist suppose \(F = F_1 \otimes F_2: X \to H \) is a phase retrieval frame for \(H\). Let \(f_1, f_2 \in H_1\) such that
\begin{equation}\label{Eq_5-1}
\left|\langle f_1, F_1(x_1) \rangle\right| = \left|\langle f_2, F_1(x_1) \rangle\right| \quad \text{for } \mu_1\text{-almost all } x_1 \in X_1.
\end{equation}
Choose any non-zero \(g \in H_2\) by (\ref{Eq_5-1}) we get
\[
\left|\langle f_1 \otimes g, F(x_1, x_2) \rangle\right| = \left|\langle f_2 \otimes g, F(x_1, x_2) \rangle\right| \quad \text{for } \mu\text{-almost all } (x_1, x_2) \in X.
\]
Due to the phase retrieval property of \(F\) in the space \(H\), it follows that \(f_1 \otimes g = e^{i\theta} (f_2 \otimes g)\) for some \(\theta \in \mathbb{R}\). This directly implies that \(f_1 = e^{i\theta} f_2\), establishing that \(F_1\) is a phase retrieval frame for \(H_1\). A parallel argument applies to \(F_2\) by considering vectors in \(H_2\) and choose any non-zero element of \(H_1\) in the tensor product. Hence, \(F_2\) is also a phase retrieval frame for \(H_2\).

Conversely, assume that \(F_1\) and \(F_2\) are phase retrieval continuous frames for \(H_1\) and \(H_2\), respectively. We show that \(F = F_1 \otimes F_2: X \to H\) is a phase retrieval frame for \(H\), where \(H = H_1 \otimes H_2\) and \(X = X_1 \times X_2\). Consider any \(T_1 , T_2 \in H_1 \otimes H_2\) such that
\[
\left|\left\langle T_1, F(x, y)\right\rangle\right| = \left|\left\langle T_2, F(x, y)\right\rangle\right| \quad \text{for } \mu\text{-almost all } (x, y) \in X.
\]
Then we get
\[
\left|\left\langle T_1, F_1(x) \otimes F_2(y) \right\rangle \right| = \left|\left\langle T_2, F_1(x) \otimes F_2(y) \right\rangle \right| \quad \text{for } \mu\text{-almost all } (x, y) \in X.
\]
Hence we have
$$
\left|\left\langle T_1(F_2 (y)), F_1(x) \right\rangle \right| = \left|\left\langle T_2( F_2(y) ), F_1(x) \right\rangle \right|
 \quad \text{for } \mu\text{-almost all } (x, y) \in X,
$$
and
$$
\left|\left\langle T_{1}^* (F_1 (x)), F_2(y) \right\rangle \right| = \left|\left\langle T_2^* ( F_1(x) ), F_2(y) \right\rangle \right|
 \quad \text{for } \mu\text{-almost all } (x, y) \in X.
$$
Since \(F_1\) and \(F_2\) does phase retrieval, there exist scalars \(\alpha_y, \beta_x\) associated to \(y\) and \(x\) respectively so that \(\left|\alpha_y\right| = \left|\beta_x\right| = 1\),
$$
T_1(F_2(y)) = \alpha_y T_2(F_2(y)) \quad \text{for } \mu_2\text{-almost all } y \in X_2,
$$
and
$$
T_{1}^* (F_1(x)) = \beta_x T_{2}^* (F_1(x)) \quad \text{for } \mu_1\text{-almost all } x \in X_1.
$$
Now for each $x \in X$ and $y \in Y$, we can write
\begin{align*}
\left\langle T_1(F_2 (y)), F_1(x) \right\rangle &= \alpha_y \left\langle T_2(F_2 (y)), F_1(x) \right\rangle  \\
&= \alpha_y \left\langle F_2 (y),T_2^*( F_1(x) )\right\rangle \\
&= \frac{\alpha_y}{\beta_x} \left\langle F_2 (y),T_1^*( F_1(x) )\right\rangle.
\end{align*}
So,
\begin{equation}\label{a_y0}
\beta_x \left\langle T_1(F_2 (y)), F_1(x) \right\rangle = \alpha_y \left\langle T_1(F_2 (y)), F_1(x) \right\rangle.
\end{equation}
If $T_1(F_2(y)) = 0$ for some $y \in X_2$, then $\beta_x$ can be any fixed unimodular constant. Let there exist  $y_0,y_1 \in X_2$ such that $T_1(F_2(y_0)) \neq 0$, $T_1(F_2(y_1)) \neq 0$
and
$$
\Delta_k = \{ x \in X_1 \; :  \; \left\langle T_1(F_2 (y_k)), F_1(x) \right\rangle \neq 0 \} \quad k = 0,1 .
$$
Also $\mu(\Delta_0), \mu(\Delta_1) > 0 $, since $F_1$ is a continuous frame so is $\mu$-complete. Moreover $\mu(\Delta_0 \cap \Delta_1) > 0$. In fact if $\mu(\Delta_0 \cap \Delta_1) = 0$, then there exist $\Gamma \subset X_1$ such that
$\Gamma \subset {\Delta_0}^c \text{ with } \Gamma^c \subset {\Delta_1}^c$ almost $\mu_1$ every where and we have
$$
\left\langle T_1(F_2 (y_0)), F_1(x) \right\rangle = 0, \quad \text{ for all } x \in \Gamma,
$$
and
$$
\left\langle T_1(F_2 (y_1)), F_1(x) \right\rangle = 0, \quad \text{ for all } x \in \Gamma^c,
$$
which is a contradiction to the part (a) of Theorem \ref{phase_equvalent_cond}, since $F_1(x)$ is phase retrieval continuous frame.

Assume $x_0 \in \Delta_0 \cap \Delta_1$ so we have
$$
\left\langle T_1(F_2 (y_0)), F_1(x_0) \right\rangle \neq 0, \quad \left\langle T_1(F_2 (y_1)), F_1(x_0) \right\rangle \neq 0.
$$
Hence by (\ref{a_y0})
$$
\alpha_{y_0} = \beta_{x_0} = \alpha_{y_1}.
$$
i.e. $\alpha_{y_0} = \alpha_{y_1}$.

\end{proof}

It is worthwhile to note that if \( F \) is a norm retrieval frame, then \( F_1 \) and \( F_2 \) are also norm retrieval frames. Indeed, let \( F_1: X_1 \to H_1 \) and \( F_2: X_2 \to H_2 \) be continuous frames for their respective Hilbert spaces. Assume that the tensor product frame \( F: X_1 \times X_2 \to H_1 \otimes H_2 \) is norm retrieval for \( H_1 \otimes H_2 \). Then \( F_1 \) is norm retrieval for \( H_1 \), and \( F_2 \) is norm retrieval for \( H_2 \). This result can be established using a similar argument as in Theorem \ref{Th_tensor_phase}. The converse of this result is established in the following theorem.

\begin{theorem}
Suppose \( F_1: X_1 \rightarrow H_1 \) is a Parseval continuous frame and \( F_2: X_2 \rightarrow H_2 \) is a norm retrieval continuous frame for their respective Hilbert spaces \(H_1\) and \(H_2\). Then the tensor product frame \( F: X_1 \times X_2 \rightarrow H_1 \otimes H_2 \) defined by \( F(x_1, x_2) = F_1(x_1) \otimes F_2(x_2) \) is norm retrieval for \(H_1 \otimes H_2\).
\end{theorem}
\begin{proof}
For any \( T_1, T_2 \in H_1 \otimes H_2 \), assume that
\[
\left|\left\langle T_1, F(x, y)\right\rangle\right| = \left|\left\langle T_2, F(x, y)\right\rangle\right| \quad \text{for all } (x, y) \in X_1 \times X_2.
\]
This implies,
\[
\left|\left\langle T_1 F_2(y), F_1(x)\right\rangle\right| = \left|\left\langle T_2 F_2(y), F_1(x)\right\rangle\right| \quad \text{for all } (x, y),
\]
which by the norm retrieval property of \( F_2 \) yields to
\[
\left\|T_1^* F_1(x)\right\| = \left\|T_2^* F_1(x)\right\| \quad \text{for a.e } x \in X_1.
\]

Assume that \( \{e_j\}_{j \in J} \)
is an orthonormal basis  of \(H_2\), it follows that
\begin{align*}
 \|T_1\|^2 = \sum_{j \in J} \left\|T_1 e_j\right\|^2 &= \sum_{j \in J}  \int_{X_1} \left|\left\langle T_1(e_j) ,F_1(x) \right\rangle\right|^2 \, d\mu_1(x) \\
 &= \int_{X_1} \sum_{j \in J} \left|\left\langle e_j , {T_1}^* F_1(x) \right\rangle\right|^2 \, d\mu_1(x) \\
 &= \int_{X_1} \| {T_1}^* F_1(x) \|^2 d\mu_1(x) \\
 &= \int_{X_1} \| {T_2}^* F_1(x) \|^2 d\mu_1(x) = \|T_2\|^2 .
\end{align*}

\end{proof}

% -----------------------------------------
\bibliographystyle{amsplain}
\bibliography{sample2}

\end{document}